\documentclass[12pt, reqno]{amsart}
\usepackage{amsmath, amsthm, amscd, amsfonts, amssymb, graphicx, color}
\usepackage[bookmarksnumbered, colorlinks, plainpages]{hyperref}

\newtheorem{theorem}{Theorem}[section]
\newtheorem{lemma}[theorem]{Lemma}
\newtheorem{proposition}[theorem]{Proposition}
\newtheorem{corollary}[theorem]{Corollary}
\theoremstyle{definition}

\theoremstyle{remark}

\numberwithin{equation}{section}

\begin{document}

\title[Reverse triangle inequality]{Reverse triangle inequality in Hilbert $C^*$-modules}
\author[M.\ Khosravi, H.\ Mahyar, M.S.\ Moslehian]{M.\ Khosravi$^1$, H.\ Mahyar$^2$ and M.\ S.\ Moslehian$^3$}
\address{Maryam Khosravi: Department of Mathematics, Teacher Training
University, Tahran, Iran.} \email{khosravi$_-$m@saba.tmu.ac.ir}
\address{Hakimeh Mahyar: Department of Mathematics, Teacher Training
University, Tahran, Iran.} \email{mahyar@saba.tmu.ac.ir}
\address{Mohammad Sal Moslehian: Department of Pure Mathematics and Centre of Excellence in Analysis on
Algebraic Structures (CEAAS), Ferdowsi University of Mashhad, P. O.
Box 1159, Mashhad 91775, Iran.} \email{moslehian@ferdowsi.um.ac.ir
and moslehian@ams.org} 

\begin{abstract}
We prove several versions of reverse triangle inequality in Hilbert
$C^*$-modules. We show that if $e_1, \cdots, e_m$ are vectors in a
Hilbert module ${\mathfrak X}$ over a $C^*$-algebra ${\mathfrak A}$
with unit 1 such that $\langle e_i,e_j\rangle=0\,\, (1\leq i\neq j
\leq m)$ and $\|e_i\|=1\,\,(1\leq i\leq m)$, and also
$r_k,\rho_k\in\Bbb{R}\,\,(1\leq k\leq m)$ and $x_1, \cdots, x_n\in
{\mathfrak X}$ satisfy
$$0\leq r_k^2 \|x_j\|\leq {Re}\langle r_ke_k,x_j\rangle\,,\quad0\leq \rho_k^2 \|x_j\|
\leq {Im}\langle \rho_ke_k,x_j\rangle\,,$$ then
\begin{eqnarray*}
\left[\sum_{k=1}^m(r_k^2+\rho_k^2)\right]^{\frac{1}{2}}\sum_{j=1}^n
\|x_j\|\leq\left\|\sum_{j=1}^nx_j\right\|\,,
\end{eqnarray*} and the
equality holds if and only if
\begin{eqnarray*}
\sum_{j=1}^n
x_j=\sum_{j=1}^n\|x_j\|\sum_{k=1}^m(r_k+i\rho_k)e_k\,.
\end{eqnarray*}
\end{abstract}
\maketitle


\section{Introduction and preliminaries}

The triangle inequality is one of the most fundamental inequalities
in mathematics. Several mathematician have been investigated its
generalizations and its reverses.\\
In 1917, Petrovitch \cite{pet} proved that for complex numbers $z_1,
\cdots, z_n$,
$$|\sum_{j=1}^n
z_j|\geq\cos\theta\sum_{j=1}^n|z_j|\,,$$ where
$0<\theta<\frac{\pi}{2}$ and $\alpha-\theta<\mbox{arg~}
z_j<\alpha+\theta \,\,
(1 \leq j \leq n)$ for a given real number $\alpha$.\\
The first generalization of the reverse triangle inequality in
Hilbert spaces was given by Diaz and Matcalf \cite{diaz}. They
proved that for $x_1, \cdots, x_n$ in a Hilbert space $H$, if $e$ is
a unit vector of $H$ such that $0\leq r\leq\frac{{\rm Re}\langle
x_j,e\rangle}{\|x_j\|}$ for some $r\in\mathbb{R}$ and each $1 \leq j
\leq n$, then
$$r\sum_{j=1}^n\|x_j\|\leq\|\sum_{j=1}^n x_j\|\,.$$
Moreover, the equality holds if and only if $\sum_{j=1}^n
x_j=r\sum_{j=1}^n\|x_j\|e$.

Recently, a number of mathematicians have represented several
refinements of the reverse triangle inequality in Hilbert spaces and
normed spaces. See \cite{moslehian, moslehian1, B-D-H-P, DRA1, DRA2,
M-S-T, M-P-F, N-T}. Recently A discussion of $C^*$-valued triangle
inequality in Hilbert $C^*$-modules was given in \cite{A-R}. Our aim is to give some generalizations of
results of Dragomir in Hilbert spaces to the framework of Hilbert
$C^*$-modules. For this purpose, we first recall some fundamental
definitions in the theory of Hilbert $C^*$-modules.

Suppose that ${\mathfrak A}$ is a $C^*$-algebra and ${\mathfrak X}$
is a linear space, which is an algebraic right ${\mathfrak
A}$-module. The space ${\mathfrak X}$ is called a pre-Hilbert
${\mathfrak A}$-module (or an inner product ${\mathfrak A}$-module)
if there exists an ${\mathfrak A}$-valued inner product
$\langle.,.\rangle:{\mathfrak X}\times {\mathfrak X}\to
{\mathfrak A}$ with the following properties:\\
(i) $\langle x,x\rangle\geq0$ and $\langle x,x\rangle=0$ if and
only if $x=0$\\
(ii) $\langle x,\lambda y+z\rangle=\lambda\langle
x,y\rangle+\langle x,z\rangle$\\
(iii) $\langle x,ya\rangle=\langle x,y\rangle a$\\
(iv) $\langle x,y\rangle^*=\langle y,x\rangle$\\
for all $x,y,z\in {\mathfrak X},\ a\in {\mathfrak A},\
\lambda\in\mathbb{C}$. By (ii) and (iv), $\langle.,.\rangle$ is
conjugate linear in the first variable. Using the Cauchy--Schwartz
inequality $\langle y,x\rangle\langle x,y\rangle\leq\|\langle
x,x\rangle\| \langle y,y\rangle$ \cite[Page 5]{lance}, it follows
that $\|x\|=\|\langle x,x\rangle\|^{\frac{1}{2}}$ is a norm on
${\mathfrak X}$ making it into a right normed module. The
pre-Hilbert module ${\mathfrak X}$ is called a Hilbert ${\mathfrak
A}$-module if it is complete with respect to this norm. Notice that
the inner structure of a $C^*$-algebra is essentially more
complicated than complex numbers. For instance, the notations such
as orthogonality and theorems such as Riesz' representation in the
complex Hilbert space theory cannot simply be generalized or
transferred to the theory of Hilbert $C^*$-modules.

One may define an ``${\mathfrak A}$-valued norm'' $|.|$ by
$|x|=\langle x,x\rangle^{1/2}$. Clearly, $\|\,|x|\,\|=\|x\|$ for
each $x\in {\mathfrak X}$. It is known that $|.|$ does not satisfy
the triangle inequality in general. See \cite{lance, M-T} for more
information on Hilbert $C^*$-modules.

We also use the elementary $C^*$-algebra theory, in particular we
utilize this property that if $a \leq b$ then $a^{1/2} \leq
b^{1/2}$, where $a,b$ are positive elements of a $C^*$-algebra
${\mathfrak A}$. We also repeatedly apply the following known
relation:
$$\frac{1}{2}(aa^*+a^*a)=(\mbox{Re}~a)^2 +(\mbox{Im}~a)^2\,,
\eqno{(\diamond)}$$
where $a$ is an arbitrary element of ${\mathfrak
A}$. For details on $C^*$-algebra theory we referred the readers to
\cite{MUR}.

Throughout this paper, we assume that ${\mathfrak A}$ is a unital
$C^*$-algebra with unit $1$ and for every $\lambda\in\Bbb{C}$, we
write $\lambda1$ as $\lambda$.


\section{Multiplicative reverse of the triangle inequality}

Utilizing some $C^*$-algebraic techniques we present our first
result as a generalization of \cite[Theorem 2.3]{DRA1}.


\begin{theorem}\label{1.1}
Let ${\mathfrak A}$ be a $C^*$-algebra, let ${\mathfrak X}$ be a
Hilbert ${\mathfrak A}$-module and let $x_1, \cdots, x_n\in
{\mathfrak X}$. If there exist real numbers $k_1,k_2\geq0$ with
$$0\leq k_1 \|x_j\|\leq {Re}\langle e,x_j\rangle\,,\quad 0\leq k_2 \|x_j\|
\leq {Im}\langle e,x_j\rangle\,,$$ for some $e\in {\mathfrak X}$
with $|e|\leq 1$ and all $1\leq j\leq n$, then
\begin{eqnarray} \label{ine1}
(k_1^2+k_2^2)^{\frac{1}{2}}\sum_{j=1}^n
\|x_j\|\leq\left\|\sum_{j=1}^nx_j\right\|\,.
\end{eqnarray}
\end{theorem}
\begin{proof}
Applying the Cauchy--Schwarz inequality, we get
$$|\langle e,\sum_{j=1}^n x_j\rangle|^2\leq
\|e\|^2\left|\sum_{j=1}^nx_j\right|^2 \leq
\left\|\sum_{j=1}^nx_j\right\|^2\,,$$ and
$$|\langle \sum_{j=1}^n
x_j,e\rangle|^2\leq\left\|\sum_{j=1}^nx_j\right\|^2|e|^2 \leq
\left\|\sum_{j=1}^nx_j\right\|^2\,,$$ whence
$$\begin{array}{ll}
\left\|\sum_{j=1}^nx_j\right\|^2&\geq\frac{1}{2}\left(|\langle
e,\sum_{j=1}^nx_j\rangle|^2+|\langle
\sum_{j=1}^nx_j,e\rangle|^2\right)\\
&=\frac{1}{2}\left(\langle e,\sum_{j=1}^nx_j\rangle^*\langle
e,\sum_{j=1}^nx_j\rangle+\langle \sum_{j=1}^nx_j,e\rangle^*\langle
\sum_{j=1}^nx_j,e\rangle\right)\\
&= (\mbox{Re}\langle e,\sum_{j=1}^nx_j\rangle)^2+
(\mbox{Im}\langle e,\sum_{j=1}^nx_j\rangle)^2 \qquad \big({\rm by~} (\diamond)\big)\\
&= (\mbox{Re} \sum_{j=1}^n\langle e,x_j\rangle)^2+
(\mbox{Im} \sum_{j=1}^n\langle e,x_j\rangle)^2\\
&\geq k_1^2(\sum_{j=1}^n\|x_j\|)^2+k_2^2(\sum_{j=1}^n\|x_j\|)^2\\
&=(k_1^2+k_2^2)(\sum_{j=1}^n\|x_j\|)^2\,.
\end{array}$$
\end{proof}

Using the same argument as in the proof of Theorem \ref{1.1} one can
obtain the following result, where $k_1,k_2$ are hermitian elements
of ${\mathfrak A}$.

\begin{theorem}
If the vectors $x_1, \cdots, x_n\in {\mathfrak X}$ satisfy the
conditions
$$0\leq k_1^2 \|x_j\|^2\leq ({Re}\langle e,x_j\rangle)^2\,,\quad 0\leq
k_2^2 \|x_j\|^2 \leq ({Im}\langle e,x_j\rangle)^2\,,$$ for some
hermitian elements $k_1,k_2$ in ${\mathfrak A}$, some
$e\in{\mathfrak X}$ with $|e|\leq 1$ and all $1\leq j\leq n$ then
the inequality {\rm \ref{ine1}} holds.
\end{theorem}
One may observe an integral version of inequality (\ref{ine1}) as
follows:

\begin{corollary}
Suppose that ${\mathfrak X}$ is a Hilbert ${\mathfrak A}$-module and
$f: [a,b] \to {\mathfrak X}$ is strongly measurable such that the
Lebesgue integral $\int_a^b\|f(t)\|dt$ exists and is finite. If
there exist self-adjoint elements $a_1,a_2$ in ${\mathfrak A}$ with
$$a_1^2\|f(t)\|^2\leq \mbox{Re}\langle f(t),e\rangle^2\,,\quad
a_2^2\|f(t)\|^2 \leq \mbox{Im}\langle f(t),e\rangle^2 \quad ({\rm
~a.e.~} t\in[a,b])\,,$$
where $e\in {\mathfrak X}$ with $|e|\leq1$,
then
$$(a_1^2+a_2^2)^{\frac{1}{2}}\int_a^b\|f(t)\|dt\leq\|\int_a^bf(t)dt\|\,.$$
\end{corollary}
Now we prove a useful lemma, which is frequently applied in the next
theorems (see also \cite{A-R}).
\begin{lemma}\label{Cauchy}
Let ${\mathfrak X}$ be a Hilbert ${\mathfrak A}$-module  and let
$x,y\in{\mathfrak X}$. If $|\langle x,y\rangle|=\|x\|\|y\|$, then
$$y=\frac{x\langle x,y\rangle}{\|x\|^2}\,.$$
\end{lemma}
\begin{proof}
For $x,y \in {\mathfrak X}$ we have
$$\begin{array}{ll}
0\leq \Big|y-\frac{x\langle x,y\rangle} {\|x\|^2}\Big|^2 &=\langle
y-\frac{x\langle x,y\rangle}{\|x\|^2},y-\frac{x\langle
x,y\rangle}{\|x\|^2}\rangle\\[2mm]
&=\langle y,y\rangle-\frac{1}{\|x\|^2}\langle y,x\rangle\langle
x,y\rangle+\frac{1}{\|x\|^4}\langle y,x\rangle\langle
x,x\rangle\langle x,y\rangle-\frac{1}{\|x\|^2}\langle
y,x\rangle\langle x,y\rangle\\[2mm]
&\leq|y|^2-\frac{1}{\|x\|^2}|\langle
x,y\rangle|^2=|y|^2-\frac{1}{\|x\|^2}\|x\|^2\|y\|^2\\[2mm]
&=|y|^2-\|y\|^2\leq0\,,
\end{array}$$
whence $\Big|y-\frac{x\langle x,y\rangle} {\|x\|^2}\Big|=0$. Hence
$y=\frac{x\langle x,y\rangle} {\|x\|^2}$.
\end{proof}
Using the Cauchy--Schwarz inequality, we have the following theorem
for Hilbert modules, which is similar to \cite[Theorem
2.5]{moslehian}.
\begin{theorem}
Let $e_1, \cdots, e_m$ be a family of vectors in a Hilbert module
${\mathfrak X}$ over a $C^*$-algebra ${\mathfrak A}$ such that
$\langle e_i,e_j\rangle=0\,\, (1\leq i\neq j\leq m)$ and
$\|e_i\|=1\,\, (1\leq i\leq m)$. Suppose that
$r_k,\rho_k\in\Bbb{R}\,\, (1\leq k\leq m)$ and that the vectors
$x_1, \cdots, x_n\in {\mathfrak X}$ satisfy
$$0\leq r_k^2 \|x_j\|\leq {Re}\langle r_ke_k,x_j\rangle\,,\quad0\leq \rho_k^2 \|x_j\|
\leq {Im}\langle \rho_ke_k,x_j\rangle\,,$$ Then
\begin{eqnarray} \label{ine1'}
\left[\sum_{k=1}^m(r_k^2+\rho_k^2)\right]^{\frac{1}{2}}\sum_{j=1}^n
\|x_j\|\leq\left\|\sum_{j=1}^nx_j\right\|\,,\end{eqnarray} and the
equality holds if and only if
\begin{eqnarray} \label{eq1'}
\sum_{j=1}^n
x_j=\sum_{j=1}^n\|x_j\|\sum_{k=1}^m(r_k+i\rho_k)e_k\,.
\end{eqnarray}
\end{theorem}
\begin{proof}
There is nothing to prove if $\sum_{k=1}^m(r_k^2+\rho_k^2)=0$.
Assume that $\sum_{k=1}^m(r_k^2+\rho_k^2)\neq0$. From the hypothesis
we have
$$\begin{array}{ll}
\big(\sum_{k=1}^m(r_k^2+\rho_k^2)\big)^2\big(\sum_{j=1}^n\|x_j\|\big)^2
&\leq \big(\mbox{Re}\langle\sum_{k=1}^m r_ke_k,\sum_{j=1}^n
x_j\rangle+
\mbox{Im}\langle\sum_{k=1}^m \rho_ke_k,\sum_{j=1}^n x_j\rangle\big)^2\\
&=\big(\mbox{Re}\langle\sum_{j=1}^n x_j,\sum_{k=1}^m
(r_k+i\rho_k)e_k\rangle\big)^2 \\
&\quad\quad \Big({\rm by~} \mbox{Im}(a)=\mbox{Re}(ia^*), \mbox{Re}(a^*)=\mbox{Re}(a) \quad (a \in {\mathfrak A})\Big)\\
&\leq\frac{1}{2}|\langle\sum_{j=1}^n x_j,\sum_{k=1}^m
(r_k+i\rho_k)e_k\rangle|^2\\
&\quad+
|\langle\sum_{k=1}^m (r_k+i\rho_k)e_k,\sum_{j=1}^n x_j\rangle|^2 \qquad \big({\rm by~} (\diamond)\big)\\
&\leq\frac{1}{2}\|\sum_{j=1}^n x_j\|^2|\sum_{k=1}^m
(r_k+i\rho_k)e_k|^2\\
&\quad+
\|\sum_{k=1}^m (r_k+i\rho_k)e_k\|^2|\sum_{j=1}^n x_j|^2\\
&\leq \|\sum_{j=1}^n x_j\|^2\left \|\sum_{k=1}^m(r_k+i\rho_k)e_k\right\|^2\\
&\quad\quad \big({\rm by~} |a| \leq \|a\|\quad (a\in {\mathfrak A})\big)\\
&=\|\sum_{j=1}^n x_j\|^2 \left\|\langle\sum_{k=1}^m
(r_k+i\rho_k)e_k,\sum_{k=1}^m(r_k+i\rho_k)e_k\rangle\right\|\\
&= \|\sum_{j=1}^n x_j\|^2\sum_{k=1}^m
|r_k+i\rho_k|^2\|e_k\|^2\\
&=\|\sum_{j=1}^n x_j\|^2\sum_{k=1}^m (r_k^2+\rho_k^2)\,.
\end{array}$$
Hence
$$\left[\sum_{k=1}^m(r_k^2+\rho_k^2)\right](\sum_{j=1}^n
\|x_j\|)^2\leq\left\|\sum_{j=1}^nx_j\right\|^2\,.$$ By taking
square roots the desired result follows.

Clearly we have equality in \ref{ine1'} if condition \ref{eq1'}
holds. To see the converse, first note that if equality holds in
\ref{ine1'}, then all inequalities in the relations above should be
equality. Therefore
$$r_k^2 \|x_j\|={Re}\langle
r_ke_k,x_j \rangle\,,\quad \rho_k^2 \|x_j\| = {Im}\langle
\rho_ke_k,x_j\rangle\,,$$
$$\mbox{Re}\langle\sum_{j=1}^n x_j,
\sum_{k=1}^m (r_k+i\rho_k)e_k\rangle=\langle\sum_{j=1}^n
x_j,\sum_{k=1}^m (r_k+i\rho_k)e_k\rangle\,,$$ and also
$$|\langle\sum_{k=1}^m (r_k+i\rho_k)e_k, \sum_{j=1}^n x_j\rangle|=
\|\sum_{j=1}^n x_j\|\|\sum_{k=1}^m (r_k+i\rho_k)e_k\|\,.$$ From
Lemma \ref{Cauchy} and these equalities we have
$$\begin{array}{ll}
\sum_{j=1}^n x_j&=\frac{\sum_{k=1}^m
(r_k+i\rho_k)e_k}{\|\sum_{k=1}^m
(r_k+i\rho_k)e_k\|^2}\langle\sum_{k=1}^m (r_k+i\rho_k)e_k,
\sum_{j=1}^n x_j\rangle\\[3mm]
&=\frac{\sum_{k=1}^m (r_k+i\rho_k)e_k}{\sum_{k=1}^m
(r_k^2+\rho_k^2)}\mbox{Re}\langle\sum_{k=1}^m (r_k+i\rho_k)e_k,
\sum_{j=1}^n x_j\rangle\\[3mm]
&=\frac{\sum_{k=1}^m (r_k+i\rho_k)e_k}{\sum_{k=1}^m
(r_k^2+\rho_k^2)}\sum_{k=1}^m
\sum_{j=1}^n (r_k^2\|x_j\|+\rho_k^2\|x_j\|)\\[3mm]
&=\sum_{j=1}^n\|x_j\|\sum_{k=1}^m(r_k+i\rho_k)e_k\,,
\end{array}$$
which is the desired result.
\end{proof}
In the next results of this section, we assume that ${\mathfrak
X}$ is a right Hilbert ${\mathfrak A}$-module, which is an
algebraic left $A$-module subject to
$$\langle x,ay\rangle=a\langle x,y\rangle \quad(x,y \in {\mathfrak
X}, a \in {\mathfrak A})\,.\eqno(\dagger)$$ For example if
${\mathfrak A}$ is a unital C$^*$-algebra and ${\mathfrak I}$ is a
commutative right ideal of ${\mathfrak A}$, then ${\mathfrak I}$ is
a right Hilbert module over ${\mathfrak A}$ and
$$\langle x,ay\rangle=x^*(ay)=ax^*y=a\langle x,y\rangle\quad
(x,y\in{\mathfrak I},\ a\in{\mathfrak A})\,.$$ The next theorem is
a refinement of \cite[Theorem 2.1]{DRA1}. To prove it we need the
following lemma.
\begin{lemma}
Let ${\mathfrak X}$ be a Hilbert ${\mathfrak A}$-module and $e_1,
\cdots, e_n\in {\mathfrak X}$ be a family of vectors such that
$\langle e_i,e_j\rangle=0\,\,(i\neq j)$ and $\|e_i\|=1$. If $x\in
{\mathfrak X}$, then
$$|x|^2\geq\sum_{k=1}^n|\langle e_k,x\rangle|^2\quad \mbox{and}
\quad|x|^2\geq\sum_{k=1}^n|\langle x,e_k\rangle|^2\,.$$
\end{lemma}
\begin{proof}
The first result follows from the following inequality:
$$\begin{array}{ll}
0 \leq |x-\sum_{k=1}^ne_k\langle e_k,x\rangle|^2&=\langle
x-\sum_{k=1}^n e_k\langle
e_k,x\rangle ,x-\sum_{j=1}^n e_j\langle e_j,x\rangle \rangle\\
&=\langle x,x\rangle+\sum_{k=1}^n\sum_{j=1}^n\langle
e_k,x\rangle^*\langle
e_k,e_j\rangle\langle e_j,x\rangle-2\sum_{k=1}^n|\langle e_k,x\rangle|^2\\
&=\langle x,x\rangle+\sum_{k=1}^n\langle e_k,x\rangle^*\langle
e_k,e_k\rangle\langle e_k,x\rangle-2\sum_{k=1}^n|\langle e_k,x\rangle|^2\\
&\leq |x|^2+\sum_{k=1}^n\langle e_k,x\rangle^*\langle
e_k,x\rangle-2\sum_{k=1}^n|\langle
e_k,x\rangle|^2\\
&=|x|^2-\sum_{k=1}^n|\langle
e_k,x\rangle|^2\,.\\
\end{array}$$
By considering $|x-\sum_{k=1}^n\langle e_k,x\rangle e_k|^2$,
similarly, we get the second one.
\end{proof}
Now we are able to prove the next theorem without using the
Cauchy--Schwarz inequality.

\begin{theorem}\label{2.5}
Let  $e_1, \cdots, e_m\in {\mathfrak X}$ be a family of vectors with
$\langle e_i,e_j\rangle=0\,\, (1\leq i\neq j\leq m)$ and
$\|e_i\|=1\,\,(1\leq i \leq m)$. If the vectors $x_1, \cdots, x_n\in
{\mathfrak X}$ satisfy the conditions
\begin{eqnarray}\label{azar}
0\leq r_k \|x_j\|\leq {Re}\langle e_k,x_j\rangle\,,\quad 0\leq
\rho_k \|x_j\| \leq {Im}\langle e_k,x_j\rangle\quad (1\leq j\leq
n, 1\leq k\leq m)\,,
\end{eqnarray} where $r_k,\rho_k\in [0, \infty)\,\,(1\leq k\leq
m)$, then
\begin{eqnarray}\label{2}
\left[\sum_{k=1}^m(r_k^2+\rho_k^2)\right]^{\frac{1}{2}}\sum_{j=1}^n
\|x_j\|\leq\left|\sum_{j=1}^nx_j\right|\,.
\end{eqnarray}
\end{theorem}
\begin{proof}
Applying the previous lemma for $x=\sum_{j=1}^nx_j$, we obtain
$$\begin{array}{ll}
\left|\sum_{j=1}^nx_j\right|^2&\geq\frac{1}{2}\left(\sum_{k=1}^m|\langle
e_k,\sum_{j=1}^nx_j\rangle|^2+\sum_{k=1}^m|\langle
\sum_{j=1}^nx_j,e_k\rangle|^2\right)\\
&=\sum_{k=1}^m\frac{1}{2}\left(\langle
e_k,\sum_{j=1}^nx_j\rangle^*\langle
e_k,\sum_{j=1}^nx_j\rangle+\langle
\sum_{j=1}^nx_j,e_k\rangle^*\langle
\sum_{j=1}^nx_j,e_k\rangle\right)\\
&=\sum_{k=1}^m (\mbox{Re}\langle e_k,\sum_{j=1}^nx_j\rangle)^2+
(\mbox{Im}\langle e_k,\sum_{j=1}^nx_j\rangle)^2 \qquad \big({\rm by~} (\diamond)\big)\\
&=\sum_{k=1}^m (\mbox{Re} \sum_{j=1}^n\langle e_k,x_j\rangle)^2+
(\mbox{Im} \sum_{j=1}^n\langle e_k,x_j\rangle)^2\\
&\geq\sum_{k=1}^m(r_k^2(\sum_{j=1}^n\|x_j\|)^2+\rho_k^2(
\sum_{j=1}^n\|x_j\|)^2) \qquad \big({\rm by~ (\ref{azar})}\big)\\
&=\sum_{k=1}^m(r_k^2+\rho_k^2)(\sum_{j=1}^n\|x_j\|)^2\,.
\end{array}$$
\end{proof}
\begin{proposition}
In Theorem {\rm \ref{2.5}}, if $\langle e_k,e_k\rangle=1$, then
the equality holds in {\rm (\ref{2})} if and only if
\begin{eqnarray}\label{3}
\sum_{j=1}^n x_j=(\sum_{j=1}^n\|x_j\|)\sum_{k=1}^m
(r_k+i\rho_k)e_k\,.
\end{eqnarray}
\end{proposition}
\begin{proof}
If (\ref{3}) holds, then inequality in (\ref{2}) turns trivially
into equality.

Next, assume that equality holds in (\ref{2}). Then two
inequalities in the proof of Theorem \ref{2.5} should be equality.
Hence
$$|\sum_{j=1}^n x_j|^2=\sum_{k=1}^m|\langle e_k,\sum_{j=1}^n x_j
\rangle|^2\quad \mbox{and}\quad|\sum_{j=1}^n
x_j|^2=\sum_{k=1}^m|\langle \sum_{j=1}^n x_j,e_k\rangle|^2\,,$$
which is equivalent to
$$\sum_{j=1}^n x_j=\sum_{k=1}^m\sum_{j=1}^ne_k\langle e_k,x_j\rangle
=\sum_{k=1}^m\sum_{j=1}^n\langle e_k,x_j\rangle e_k\,,$$ and also
$$r_k \|x_j\|=\mbox{Re}\langle e_k,x_j\rangle\,,\quad\rho_k \|x_j\|
=\mbox{Im}\langle e_k,x_j\rangle\,.$$ So
$$\sum_{j=1}^n x_j=\sum_{k=1}^m\sum_{j=1}^ne_k\langle e_k,x_j
\rangle=\sum_{k=1}^m\sum_{j=1}^ne_k(r_k+i\rho_k)\|x_j\|=
(\sum_{j=1}^n\|x_j\|)\sum_{k=1}^m (r_k+i\rho_k)e_k\,.$$
\end{proof}


\section{Additive reverse of the triangle inequality}

There are some versions of additive reverse of triangle inequality.
In \cite{drbook}, Dragomir established the following theorem:
\begin{theorem} Let $\{e_k\}_{k=1}^m$ be a family of orthonormal vectors in Hilbert
space $H$ and ${M_{jk}\geq0\,\,(1\leq j\leq n,1\leq k\leq m)}$
such that
$$\|x_j\|-\mbox{Re}\langle e_k,x_j\rangle\leq M_{jk}\,,$$
for each $1\leq j\leq n$ and $1\leq k\leq m$. Then
$$\sum_{j=1}^n\|x_j\|\leq\frac{1}{\sqrt{m}}\|
\sum_{j=1}^nx_j\|+\frac{1}{m} \sum_{j=1}^n\sum_{k=1}^m M_{jk}\,;$$
and the equality holds if and only if
$$\sum_{j=1}^n\|x_j\|\geq\frac{1}{m}\sum_{j=1}^n\sum_{k=1}^m M_{jk}\,,$$
and
$$\sum_{j=1}^nx_j=\left(\sum_{j=1}^n\|x_j\|-\frac{1}{m}
\sum_{j=1}^n\sum_{k=1}^m M_{jk}\right)\sum_{k=1}^me_k\,.$$
\end{theorem}
We can prove this theorem for Hilbert $C^*$-modules with some
different techniques.
\begin{theorem}
Let  $\{e_k\}_{k=1}^m$ be a family of vectors in Hilbert module
${\mathfrak X}$ over a $C^*$-algebra ${\mathfrak A}$ with unit $1$
with $|e_k|\leq 1 \,\,(1\leq k \leq m)$ and $\langle e_i,
e_j\rangle=0\,\,(1\leq i\neq j\leq m)$ and $x_j\in{\mathfrak
X}\,\, (1\leq j \leq n)$. If for some scalars
$M_{jk}\geq0\,\,(1\leq j\leq n$, $1\leq k\leq m)$,
\begin{eqnarray}\label{azar2}
\|x_j\|-\mbox{Re}\langle e_k,x_j\rangle\leq M_{jk} \qquad (1\leq
j\leq n, 1\leq k\leq m)\,,
\end{eqnarray}
then
\begin{eqnarray}\label{add}\sum_{j=1}^n\|x_j\|\leq\frac{1}{\sqrt{m}}\|
\sum_{j=1}^nx_j\|+\frac{1}{m} \sum_{j=1}^n\sum_{k=1}^m M_{jk}\,.
\end{eqnarray}
Moreover, if $|e_k|=1 \,\,(1\leq k \leq m)$, then the equality in
{\rm (\ref{add})} holds if and only if
\begin{eqnarray}\label{add1}
\sum_{j=1}^n\|x_j\|\geq\frac{1}{m}\sum_{j=1}^n \sum_{k=1}^m
M_{jk}\,,
\end{eqnarray}
and
\begin{eqnarray}\label{add2}\sum_{j=1}^nx_j=\left(\sum_{j=1}^n\|x_j\|-\frac{1}{m}
\sum_{j=1}^n\sum_{k=1}^m
M_{jk}\right)\sum_{k=1}^me_k\,.\end{eqnarray}
\end{theorem}
\begin{proof}
Taking the summation in (\ref{azar2}) over $j$ from 1 to $n$, we
obtain
$$\sum_{j=1}^n\|x_j\|\leq\mbox{Re}\langle e_k,\sum_{j=1}^nx_j\rangle
+\sum_{j=1}^n M_{jk}\,,$$ for each $k\in\{1,\cdots,m\}$. Summing
these inequalities over $k$ from 1 to $m$, we deduce
\begin{eqnarray}\label{2.9}\sum_{j=1}^n\|x_j\|\leq\frac{1}{m}\mbox{Re}\langle
\sum_{k=1}^me_k, \sum_{j=1}^nx_j\rangle+\frac{1}{m}
\sum_{k=1}^m\sum_{j=1}^n M_{jk}\,.\end{eqnarray} Using the
Cauchy--Schwarz we obtain
\begin{eqnarray}\label{2.10}
\begin{array}{ll}
\Big(\mbox{Re}\langle \sum_{k=1}^me_k,
\sum_{j=1}^nx_j\rangle\Big)^2&\leq \frac{1}{2}(|\langle
\sum_{k=1}^me_k, \sum_{j=1}^nx_j\rangle|^2+
|\langle \sum_{k=1}^me_k, \sum_{j=1}^nx_j\rangle^*|^2)\quad \big({\rm by~} (\diamond)\big)\\
&\leq\frac{1}{2}(\|\sum_{k=1}^me_k\|^2 |\sum_{j=1}^nx_j|^2+
|\sum_{k=1}^me_k|^2\| \sum_{j=1}^nx_j\|^2)\\
&\leq\|\sum_{k=1}^me_k\|^2\| \sum_{j=1}^nx_j\|^2\\
&\leq m\| \sum_{j=1}^nx_j\|^2\,,
\end{array}
\end{eqnarray}
since
$$\|\sum_{k=1}^me_k\|^2=\|\langle\sum_{k=1}^me_k,\sum_{k=1}^me_k\rangle\|=
\|\sum_{k=1}^m\sum_{l=1}^m\langle
e_k,e_l\rangle\|=\|\sum_{k=1}^m|e_k|^2\|\leq m\,.$$ Using
(\ref{2.10}) in (\ref{2.9}), we deduce the desired inequality.

If (\ref{add1}) and (\ref{add2}) hold, then
$$\begin{array}{ll}
\frac{1}{\sqrt{m}}\left\|\sum_{j=1}^nx_j\right\|&=\frac{1}{\sqrt{m}}\left(
\sum_{j=1}^n\|x_j\|
-\frac{1}{m}\sum_{j=1}^n\sum_{k=1}^mM_{jk}\right)
\|\sum_{k=1}^me_k\|\\
&=\sum_{j=1}^n\|x_j\|-\frac{1}{m}\sum_{j=1}^n
\sum_{k=1}^mM_{jk}\,,
\end{array}$$
and then the equality in \ref{add} holds true.

Conversely, if the equality holds in (\ref{add}), then obviously
(\ref{add1}) is valid and we have equalities all over in the proof
above. This means that
$$\|x_j\|-\mbox{Re}\langle e_k,x_j\rangle= M_{jk}\,,$$
$$\mbox{Re}\langle \sum_{k=1}^me_k, \sum_{j=1}^nx_j\rangle=\langle
\sum_{k=1}^me_k,\sum_{j=1}^nx_j\rangle\,,$$ and
$$|\langle \sum_{k=1}^me_k,
\sum_{j=1}^nx_j\rangle|=\|\sum_{k=1}^me_k\| \|\sum_{j=1}^nx_j\|\,.$$
It follows from Lemma \ref{Cauchy} and the previous relations that
$$\begin{array}{ll}
\sum_{j=1}^nx_j&=\frac{\sum_{k=1}^me_k}{\|\sum_{k=1}^me_k\|^2}
\langle \sum_{k=1}^me_k,\sum_{j=1}^nx_j\rangle\\[2mm]
&=\frac{\sum_{k=1}^me_k}{m}\mbox{Re}
\langle \sum_{k=1}^me_k,\sum_{j=1}^nx_j\rangle\\[2mm]
&=\frac{\sum_{k=1}^me_k}{m}\sum_{k=1}^m\sum_{j=1}^n(\|x_j\|-M_{jk})\\[2mm]
&=\left(\sum_{j=1}^n\|x_j\|-\frac{1}{m} \sum_{j=1}^n\sum_{k=1}^m
M_{jk}\right)\sum_{k=1}^me_k\,.
\end{array}$$
\end{proof}


\end{document}